\documentclass{amsart}

\newtheorem{theorem}{Theorem}
\newtheorem{proposition}[theorem]{Proposition}
\newtheorem{corollary}[theorem]{Corollary}
\theoremstyle{remark}
\newtheorem{remark}{Remark}

\usepackage{graphicx} 
\usepackage{overpic}
\usepackage{xcolor}

\newcommand{\sign}{\mathop{\mathrm{sign}}\nolimits}
\newcommand{\lk}{\mathop{\mathrm{lk}}\nolimits}

\begin{document}

\title[Cosmetic surgeries on knots in homology spheres]{Cosmetic surgeries on knots in homology spheres and the Casson--Walker invariant}

\author{Kazuhiro Ichihara}
\address{Department of Mathematics, College of Humanities and Sciences, Nihon University, 3-25-40 Sakurajosui, Setagaya-ku, Tokyo 156-8550, JAPAN}
\email{ichihara.kazuhiro@nihon-u.ac.jp}

\author{In Dae Jong}
\address{Department of Mathematics, Kindai University, 3-4-1 Kowakae, Higashiosaka City, Osaka 577-0818, Japan} 
\email{jong@math.kindai.ac.jp}


\date{\today}

\subjclass[2020]{Primary 57K30, Secondary 57K10}

\keywords{cosmetic surgery, homology $3$-sphere, Casson--Walker invariant, Casson--Gordon invariant, rational surgery formula}

\thanks{This work was supported by JSPS KAKENHI Grant Numbers JP22K03301 and JP22K03324.}

\begin{abstract}
We study cosmetic surgeries on a knot in a homology sphere. 
Several constraints on knots and surgery slopes to admit such surgeries are given. Our main ingredient is the rational surgery formula of the Casson--Walker invariant for 2-component links in the 3-sphere. 
\end{abstract}

\maketitle

\section{Introduction}

One of current active research problem in Dehn surgery theory is the cosmetic surgery problem, which asks under what circumstances two distinct Dehn surgeries on the same knot yield homeomorphic 3-manifolds.
A pair of such surgeries is referred to as \textit{cosmetic surgeries}.
More specifically, they are called \textit{purely cosmetic} if the resulting manifolds are homeomorphic via an orientation-preserving map, and \textit{chirally cosmetic} if the homeomorphism reverses orientation.
It is conjectured that no nontrivial knot in a closed, oriented 3-manifold admits purely cosmetic surgeries along inequivalent slopes.
This is known as the \textit{purely cosmetic surgery conjecture}.
For further details, see \cite{BleilerHodgsonWeeks} and \cite[Problem 1.81A]{Kirby}.

In this paper, we study cosmetic surgeries on knots in homology spheres.
We first focus on the purely cosmetic case.
For a knot in an integral homology sphere, the following result holds as a corollary of a theorem by Boyer and Lines~\cite[Theorem 2.8]{BoyerLines}.

\begin{proposition}\label{prop:BL1}
   For any knot in an integral homology sphere, there are at most two pairs of integral purely cosmetic surgeries. 
\end{proposition}

Here Dehn surgery is said to be \emph{integral} if the algebraic intersection number of the meridian and a representative of the surgery slope is $\pm 1$. 
The proof of Proposition~\ref{prop:BL1} will be given later. 
The following result extends Proposition~\ref{prop:BL1} to knots in rational homology spheres obtained by Dehn surgery on knots in the 3-sphere $S^3$. 

\begin{theorem}\label{thm:rational}
    Let $K$ be a null-homologous knot in a rational homology sphere obtained by Dehn surgery on a knot in $S^3$. 
     Then $K$ admits at most two pairs of integral purely cosmetic surgeries. 
\end{theorem}

\begin{remark}
Recent studies on cosmetic surgeries on knots in $S^3$ have made extensive use of Heegaard Floer homology. 
As a recent culmination of this approach, it has been shown that there is at most one pair of purely cosmetic surgeries on any nontrivial knot in $S^3$ (directly as a corollary of \cite[Corollary 1.4]{daemi2024filteredinstantonhomologycosmetic}). 
Moreover, for several classes of knots in $S^3$, the purely cosmetic surgery conjecture has been confirmed affirmatively.
See \cite{daemi2024filteredinstantonhomologycosmetic,IchiharaJongMattmanSaito,KotelskiyLidmanMooreWatsonZibrowius,StipsiczSzabo,TaoComposite,Varvarezos21} for examples.
\end{remark}

\begin{remark}
For a hyperbolic knot in an arbitrary closed 3-manifold, it has been shown in \cite[Theorem 1.13]{FPS22} that there exists a finite set of pairs of slopes such that Dehn surgeries along any pair not in this set are never purely cosmetic. 
\end{remark}

For a knot in an integral homology sphere, Boyer and Lines also provide a sufficient condition under which no purely cosmetic surgeries can occur.

\begin{proposition}[{\cite[(5.1) Proposition (iii)]{BoyerLines}}]\label{prop:BL2}
Let $K$ be a knot in an integral homology sphere. 
If $\Delta''_K(1) \ne 0$, 
then $K$ admits no purely cosmetic surgeries. 
\end{proposition}

Here $\Delta_K(t)$ denotes the Alexander polynomial of $K$, normalized so that $\Delta_K(t^{-1}) = \Delta_K(t)$ and $\Delta_K(1) = 1$.
It is known that for a knot $K$ in $S^3$, $\Delta''_K(1)$ is equal to twice the second coefficient $a_2(K)$ of the Conway polynomial of $K$.
Thus, the following can be regarded as a partial extension of Proposition~\ref{prop:BL2} to knots in integral homology spheres obtained by Dehn surgery on knots in $S^3$.

\begin{theorem}\label{thm:integral}
    Let $\Sigma$ be an integral homology sphere obtained by Dehn surgery on a knot $K_0$ in $S^3$ along the slope $1/{q_0}$ for a non-zero integer $q_0$. 
    Let $K$ be a knot in $\Sigma$. 
    Suppose that, identifying $K$ with a knot in $S^3$ before surgery on $K_0$, the linking number of $K$ and $K_0$ in $S^3$ is zero. 
    If $a_2 (K) - q_0 a_3(L) \ne 0$, then $K$ admits no purely cosmetic surgeries. 
\end{theorem}

Here $a_3(L)$ denotes the third coefficient of the Conway polynomial of the link $L = K \cup K_0$ in $S^3$. 

Next, we focus on the chirally cosmetic case.
We denote by $\lambda_w$ the Casson--Walker invariant for homology spheres.
See \cite{Walker} for the definition and basic properties of $\lambda_w$.
The following result is also obtained from a result due to Boyer and Lines.

\begin{proposition}\label{prop:BL3}
    Let $K$ be a knot in an integral homology sphere $\Sigma$ with $\lambda_w(\Sigma) \ne 0$. 
    Then $K$ admits no integral chirally cosmetic surgeries. 
\end{proposition}

The proof of Proposition~\ref{prop:BL3} will be given later.
The following extends Proposition~\ref{prop:BL3} to knots in rational homology spheres obtained by Dehn surgery on knots in $S^3$.

\begin{theorem}\label{thm:chirally}
    Let $\Sigma = K_0(p_0)$ be a rational homology sphere obtained by $p_0$-surgery on a knot $K_0$ in $S^3$ for a non-zero integer $p_0$. 
    If $p_0 \ne \dfrac{3 \sign(p_0) \pm \sqrt{96a_2(K_0)+1}}{2}$ holds, then any null-homologous knot $K$ in $\Sigma$ admits no integral chirally cosmetic surgeries. 
\end{theorem}

The following corollary follows from Theorem~\ref{thm:chirally}  immediately. 

\begin{corollary}\label{cor:1}
Let $\Sigma = K_0(p_0)$ be a rational homology sphere obtained by $p_0$-surgery on a knot $K_0$ in $S^3$. 
If $a_2(K_0)=0$ and $|p_0| \ge 3$, then any null-homologous knot in $\Sigma$ admits no integral chirally cosmetic surgeries. 
\end{corollary}

If the exterior of a knot admits an orientation-reversing self-homeomorphism, then the knot admits integral chirally cosmetic surgeries naturally. 
Thus, from the corollary above, we have the following. 

\begin{corollary}\label{cor:amphicheiral}
Let $K_0$ be a knot in $S^3$ with $a_2(K_0) = 0$. 
For any integer $p_0 \ge 3$, the rational homology sphere $\Sigma = K_0(p_0)$ contains no null-homologous knots with the exterior admitting an orientation-reversing self-homeomorphism. 
In particular, there are no null-homologous amphicheiral knot in $\Sigma$. 
\end{corollary}

For the definition of an amphicheiral knot in a general $3$--manifold, see \cite{IchiharaJongTaniyama}. 

A key ingredient in the proofs of Theorems~\ref{thm:rational}, \ref{thm:integral}, and \ref{thm:chirally} is the rational surgery formula for the Casson--Walker invariant $\lambda_w$ for 2-component links in $S^3$, established by Ito \cite[Theorem 1.5]{Ito2022}.

\bigskip 

Here we give proofs of Propositions~\ref{prop:BL1} and \ref{prop:BL3} for completeness. 
First we prove Proposition~\ref{prop:BL1}. 
Let $K$ be a knot in an integral homology sphere $\Sigma$. 
We denote by $K (p/q)$ the 3-manifold obtained by $p/q$-surgery on $K$. 
From \cite[Theorem 2.8]{BoyerLines}, we have the following surgery formula:
\begin{equation}\label{eq:BL}
\lambda (K (p/q)) = \lambda (\Sigma) +(q/p) (\Delta_K''(1))/2 - (\mathrm{sgn} (p)/2) s(q, p).
\end{equation}
Here $\lambda$ denotes the (original) Casson invariant used in \cite{BoyerLines}, for which $\lambda_w = 2 \lambda$ holds. 
See \cite[Section 6]{Lescop} for example. 
In the above, $s(q,p)$ denotes the Dedekind sum for coprime integers $p,q$ (see the next section for the definition). 
Suppose that $K$ admits a pair of integral purely cosmetic surgeries. 
Then, by Proposition~\ref{prop:BL2}, we have $\Delta_K''(1)=0$. 
Moreover, by homological reasons, the surgery slopes must be $\pm p$ for a positive integer $p$. 
Thus, by the formula \eqref{eq:BL}, we have $s(1, p) = -s(1, -p)$. 
Since $s(1, -p) = -s(1,p)$ holds, we have $2 s(1, p) = 0$. 
By the definition of $s(q,p)$, 
\begin{equation}\label{eq:s(1,p)} 
s(1,p) = (p-1)(p-2)/(12p). 
\end{equation}
Then, we see that there are at most two values for $p$, exactly, $p = 1,2$. 
Actually, it also follows from the surgery formula of the Casson--Gordon invariant given in \cite[Lemma 2.22]{BoyerLines}. 
See the formula \eqref{CGformula} in the next section. 

Next we prove Proposition~\ref{prop:BL3}. 
By the formula \eqref{eq:BL}, for a positive integer $p$, 
\[ \lambda (K(p)) = \lambda (\Sigma) + \Delta_K''(1)/(2p) - s(1,p)/2. \]
We denote by $-M$ the $3$-manifold $M$ with opposite orientation. 
Suppose that $K(p)$ is orientation preservingly homeomorphic to $-K(-p)$. 
It suffices to show $\lambda_w (\Sigma) = 2 \lambda (\Sigma) = 0$. 
Since $\lambda (-K(p)) = -\lambda (K(p))$, we have 
\[\lambda (K(p)) + \lambda (K(-p)) =0.\]
By the formula \eqref{eq:BL}, this is equivalent to 
\[(\lambda (\Sigma) + \Delta_K''(1)/(2p) - s(1,p)/2) + (\lambda (\Sigma) - \Delta_K''(1)/(2p) + s(1,-p)/2) = 0,\]
and so, we obtain $2\lambda (\Sigma) = s(1,p)$. 
By Equation~\eqref{eq:s(1,p)}, we have $p^2 - 3\left(1 + \lambda(\Sigma) \right)p + 2 = 0$. 
If this quadratic equation in $p$ has an integer solution $n > 0$, then the other solution is $2/n$. 
Then $n + 2/n = 3\left(1 + \lambda(\Sigma) \right)$, and thus, $n = 1$ or $n=2$ since $\lambda(\Sigma) \in \mathbb{Z}$. 
This implies that $\lambda(\Sigma)=0$, which completes the proof of Proposition~\ref{prop:BL3}.

\section{Preliminary}\label{sec:preliminary}

Let $L(r,s)$ denote the 3-manifold obtained by Dehn surgeries on a 2-component link $L$ in $S^3$ along slopes $r$ and $s$ for each component. 
That is, $L(r,s)$ is obtained from the exterior $E(L) = S^3 \setminus N(L)$ of $L$ by gluing solid tori $V, V'$ so that the meridians of $V, V'$ are identified with simple closed curves of slopes $r,s$ on the boundaries $\partial E(L)$ respectively.  

A key ingredient to prove Theorems~\ref{thm:rational}, \ref{thm:integral}, and \ref{thm:chirally} is the rational surgery formula of the Casson--Walker invariant $\lambda_w$ for 2-component links in $S^3$ established in \cite[Theorem 1.5]{Ito2022}. 
See \cite{Walker} for the definition and basic properties of the Casson--Walker invariant. 
Let $L=K_x \cup K_y$ be a 2-component link in $S^3$ and $p_x / q_x$ and $p_y / q_y$ slopes for the components $K_x$ and $K_y$, respectively. 
Then, by \cite[Theorem 1.5]{Ito2022}, if $L(p_x / q_x, p_y / q_y)$ is a rational homology 3-sphere, its Casson--Walker invariant $\lambda_w=\lambda_w(L(p_x / q_x, p_y / q_y))$ is given by 
\begin{align}\label{eq:Ito}
\begin{split}
D \left(\frac{\lambda_w}{2}-\frac{\varsigma}{8}\right) &= a_2(K_x)\frac{p_y}{q_y}-\frac{p_y}{24q_y} -\frac{p_y}{24q_yq_x^2} +\frac{p_y\ell^2}{24q_y}\\
&\quad +a_2(K_y)\frac{p_x}{q_x} -\frac{p_x}{24q_x}-\frac{p_x}{24q_xq_y^2}+\frac{p_x\ell^2}{24q_x}\\
&\quad + 2v_3(L) + \frac{D}{24}\left( S\left(\frac{p_x}{q_x}\right)-\frac{p_x}{q_x} +S\left(\frac{p_y}{q_y}\right)-\frac{p_y}{q_y} \right).
\end{split}
\end{align}
The terms used in the formula~\eqref{eq:Ito} are given as follows. 
\begin{itemize}
\item $a_{i}(K)$ denotes the coefficient of $z^{i}$ in the Conway polynomial of a link $K$. 
\item $\ell = \lk(K_x,K_y)$ denotes the linking number of $L = K_x \cup K_y$, which coincides with $a_1(L)$. 
\item For the linking matrix $A=\begin{pmatrix} \frac{p_x}{q_x} & \ell \\ \ell & \frac{p_y}{q_y} \end{pmatrix}$, $D= \det A$ and $\varsigma =\sigma(A)$. Here $\sigma$ denotes the signature.
\item $v_3(L)$ denotes 
the invariant defined by 
\[ v_3(L)=\frac{1}{2}\left( -a_3(L)+(a_2(K_x)+a_2(K_y))\ell +\frac{1}{12}(\ell^3-\ell) \right) . \]
\item For coprime integers $p$ and $q$, $S\left(\dfrac{p}{q}\right)$ denotes the \emph{Dedekind symbol} related to the more famous Dedekind sum $s(p,q)$ by 
\[ S\left(\frac{p}{q}\right)=  12(\mathrm{sign}(q))s(p,q). \]
The Dedekind sum $s(p,q)$ is defined\footnote{In \cite[Theorem 1.5]{Ito2022}, there is a typo in the definition. Compare with \cite[Definition 1.5]{BarNatanLawrence}\label{ft}. Note that, not explicitly stated in \cite[Definition 1.5]{BarNatanLawrence}, $s(p,\pm 1)$ should be $0$. In addition, our notation for the Dedekind sum looks different from that in \cite{BoyerLines}, but they are essentially same. 
} by
\[ s(p,q)=\sum_{k=1}^{|q|-1} \left(\!\!\left( \frac{k}{q}\right)\!\!\right)\left(\!\!\left( \frac{kp}{q}\right)\!\!\right), 
\]
where $(\!( x)\!)=x-\lfloor x\rfloor -\frac{1}{2}$.
\end{itemize}

We note that the Dedekind symbol satisfies the following.
\begin{equation}\label{eq:DedekindSymbol}
    S\left(\frac{p}{q}\right) + S\left(\frac{q}{p}\right) = 
\frac{p}{q} + \frac{q}{p} + \frac{1}{pq} - 3 \sign (p q).
\end{equation}
See \cite[Definition 1.5]{BarNatanLawrence} for example. 

Another key ingredient is the surgery formula of the total Casson--Gordon invariant $\tau$ for homology lens spaces. 
See \cite[Definition 2.20]{BoyerLines} and \cite{CassonGordon} for the definition. 
Let $K$ be a knot in an integral homology 3-sphere $\Sigma$. 
The Casson--Gordon invariant\footnote{In \cite{BoyerLines}, this invariant is called the total Casson--Gordon invariant. In this paper, we call it without ``total'' for brevity.} $\tau$ satisfies the following surgery formula \cite[Lemma 2.22]{BoyerLines}: 
\begin{equation}\label{CGformula}
\tau(\Sigma_{K}(p/ q)) = -4p \cdot s(q,p) - \sigma(K,p).
\end{equation}
Here $\sigma(K,p)= \sum_{\omega: \omega^{p}=1}\sigma_{\omega}(K)$ denotes \textit{the total $p$-signature} for $K$, where $\sigma_{\omega}(K)$ ($\omega \in \{z \in \mathbb{C}\: | \: |z|=1 \}$) denotes \textit{the Levine-Tristram signature}, i.e., the signature of $(1-\omega)S + (1-\overline{\omega})S^{T}$ for a Seifert matrix $S$ of $K$.

\section{Proofs}

In the following, by $\cong$, we mean orientation-preservingly homeomorphic. 
Under the setting above, we have the following, from which Theorem~\ref{thm:rational} follows. 

\begin{theorem}\label{thm3}
    Let $L=K \cup K_0$ be a two-component link in $S^3$ with $\lk (K,K_0) = 0$. 
    Then, for a non-zero rational number $p_0 / q_0$, there are at most two positive integers $p$ satisfying $L(p/1,p_0 / q_0) \cong L(-p/1, p_0 / q_0)$. 
\end{theorem}

\begin{proof}
Suppose that $L(p/1,p_0/q_0) \cong L(-p/1, p_0/q_0)$, and set $\lambda_w = \lambda_w ( L(p/1,p_0/q_0))$, $\lambda_w' = \lambda_w ( L(- p/1, p_0/q_0))$. 
Then $\lambda_w = \lambda_w'$ holds. 

We apply the formula~\eqref{eq:Ito} in two ways; apply~\eqref{eq:Ito} with $p_x = p$, $q_x=1$, $p_y=p_0$, $q_y=q_0$, and $p_x = -p$, $q_x = 1$, $p_y=p_0$, $q_y=q_0$. 
By the assumption $\ell = \lk (K,K_0) = 0$, the linking matrices, $A$ and $A'$, of the two manifolds $L(p/1,p_0/q_0)$ and $L(-p/1, p_0/q_0)$ are given by
\[ A = \begin{pmatrix} p & 0 \\ 0 & \frac{p_0}{q_0} \end{pmatrix} \text{ and } A'=\begin{pmatrix} -p & 0 \\ 0 & \frac{p_0}{q_0} \end{pmatrix}\] 
respectively. 
Thus, their determinants, $D$ and $D'$, are given by 
\[ D = \frac{p p_0}{q_0} \text{ and } D' = \frac{-p p_0}{q_0} . \]
Set $\varsigma = \sigma(A)$ and $\varsigma'= \sigma(A')$. 
Then by the formula~\eqref{eq:Ito}, for the manifold $L(p/1,p_0/q_0)$ and its Casson--Walker invariant $\lambda_w$, we have 
\begin{align*}
    \frac{p p_0}{q_0}\left( \dfrac{\lambda_w}{2} - \dfrac{\varsigma}{8}\right)
    &= a_2(K)\frac{p_0}{q_0} - \frac{p_0}{12q_0} + a_2(K_0)p - \frac{p}{24} - \frac{p}{24q_0^2} \\ 
    &\quad + 2v_3(L) + \frac{pp_0}{24q_0}\left( S\left(\frac{p}{1}\right) - p + S\left(\frac{p_0}{q_0}\right) - \frac{p_0}{q_0}\right)\\ 
    \iff
    \dfrac{\lambda_w}{2} 
    &= \frac{\varsigma}{8} + \frac{a_2(K)}{p} - \frac{1}{12p} + a_2(K_0)\frac{q_0}{p_0} - \frac{q_0}{24p_0} - \frac{1}{24p_0q_0} \\ 
    &\quad + 2v_3(L)\frac{q_0}{p p_0} + \frac{1}{24}\left( S\left(\frac{p}{1}\right) - p + S\left(\frac{p_0}{q_0}\right) - \frac{p_0}{q_0}\right). 
\end{align*}
By a similar way, for the manifold $L(-p/1,p_0/q_0)$ and its Casson--Walker invariant $\lambda_w'$, we have 
\begin{align*}
    \dfrac{\lambda_w'}{2} 
    &= \frac{\varsigma'}{8} - \frac{a_2(K)}{p} + \frac{1}{12p} + a_2(K_0)\frac{q_0}{p_0} - \frac{q_0}{24p_0} - \frac{1}{24p_0q_0} \\ 
    &\quad - 2v_3(L)\frac{q_0}{p p_0} + \frac{1}{24}\left( S\left(\frac{p}{-1}\right) + p + S\left(\frac{p_0}{q_0}\right) - \frac{p_0}{q_0}\right). 
\end{align*}
Taking the difference, we obtain: 
\begin{align*}
    \begin{split}
    \frac{\lambda_w}{2} - \frac{\lambda_w'}{2} 
    &= \frac{\varsigma - \varsigma'}{8} 
     +\frac{2}{p} \left( a_2(K) - \frac{1}{12} + \frac{ 2 q_0}{p_0} \,v_3(L) \right) 
     + \frac{1}{24} \left( S\left( \frac{p}{1} \right) - S\left( \frac{p}{-1} \right) - 2 p \right). 
    \end{split}
\end{align*}    
Here, we note that 
\[
S\left(\frac{p}{1}\right) - S\left(\frac{p}{-1}\right) = 12 (s(p,1) + s(p,-1)) = 0 \] 
since $s(p,\pm1) =0$. (See the footnote on page~\pageref{ft}.)
Also note that $v_3(L) = -a_3(L)/2$ since $\ell = 0$. 
Thus, combining the condition $\lambda_w = \lambda_w'$, we conclude that
\begin{equation}\label{eq:RHS1}
p^2 - \frac{3(\varsigma - \varsigma')}{2}\cdot p  + \left( 2 - 24 a_2(K) + \frac{24  q_0}{p_0} \,a_3(L) \right) =0.  
\end{equation}
There exist at most two real values for such $p$. 
This implies that, for a rational number $p_0/q_0 \ne 0$, there are at most two positive integers $p$ satisfying $L(p/1,p_0/q_0) \cong L(-p/1, p_0/q_0)$.
\end{proof}

\begin{proof}[Proof of Theorem~\ref{thm:rational}]
Let $K$ be a null-homologous knot in a rational homology sphere $\Sigma$ obtained by Dehn surgery on a knot $K_0$ in $S^3$. 
Let $K_0^*$ be the dual knot of $K_0$ in $\Sigma$. 
Since $K$ is null-homologous, $K$ bounds an oriented compact surface $F$ in $\Sigma$. 
We may assume that $F \cap K_0$ consists of finite points $t_1, \dots, t_k$. 
Take an arc $\alpha_i$ on $F$ which connects $t_i$ and a point of $\partial F$ for $i=1,\dots,k$. 
By isotoping $F$ (and $K = \partial F$) so that $\alpha_i$ shrinks to $t_i$ for $i=1,\dots,k$, we obtain an oriented compact surface, say $F$ again, disjoint from $K_0^*$ in $\Sigma$. 
This implies that, after an isotopy in $\Sigma$, $K$ bounds an oriented compact surface in the exterior of $K_0^*$ in $\Sigma$. 
Note that such an isotopy preserves the meridian-longitude pair of $K$. 
Since $F \subset \Sigma \setminus K_0^* = S^3 \setminus K_0$, $K$ is regarded as a knot in $S^3$ with the linking number $0$ with $K_0$. 

Assume that $K$ in $\Sigma$ admits integral purely cosmetic surgeries. 
Since the meridian of $K$ in $\Sigma$ is also the meridian of $K$ regarded as a knot in $S^3$ as above, the corresponding surgery slopes in $S^3$ must be $\pm p$ with some positive integer $p$. 
Thus, applying Theorem~\ref{thm3} to this $K$ and $K_0$, we see that $K$ admits at most two pairs of integral purely cosmetic surgeries. 
\end{proof}

Similarly, Theorem~\ref{thm:integral} immediately follows from the next one.  

\begin{theorem}
    Let $L = K \cup K_0$ be a 2-component link in $S^3$ with $\lk(K,K_0)=0$. 
    If $a_2 (K) - q_0 a_3(L) \ne 0$, then $L(p/q,1/q_0) \not\cong L(p/q', 1/q_0)$ for a non-zero integer $q_0$ and distinct non-zero integers $q, q'$ coprime to $p>0$. 
\end{theorem}

\begin{proof}
Suppose that $L(p/q,1/q_0) \cong L(p/q', 1/q_0)$, and set $\lambda_w = \lambda_w ( L(p/q,1/q_0))$, $\lambda_w' = \lambda_w ( L(p/q', 1/q_0))$. 
Then, $\lambda_w = \lambda_w'$ holds, and we will deduce that $a_2 (K) - q_0 a_3(L) = 0$ holds. 
To do so, we apply the formula~\eqref{eq:Ito} in two ways; 
apply~\eqref{eq:Ito} with $p_x=p$, $q_x=q$, $p_y=1$, $q_y=q_0$, and $p_x=p$, $q_x=q'$, $p_y=1$, $q_y=q_0$. 

By the assumption $\ell=0$, the linking matrices, $A$ and $A'$, of the two manifolds $L(p/1,1/q_0)$ and $L(p/q', 1/q_0)$ are given by
\[ A = \begin{pmatrix} \frac{p}{q} & 0 \\ 0 & \frac{1}{q_0} \end{pmatrix} \text{ and } A'=\begin{pmatrix} \frac{p}{q'} & 0 \\ 0 & \frac{1}{q_0} \end{pmatrix}, \] 
respectively. 
Thus, their determinants, $D$ and $D'$, are given by 
\[ D = \frac{p}{q\,q_0} \text{ and } D'=\frac{p}{q'\,q_0} , \]
respectively. 
Also, for their signatures, we have
\[
\varsigma = 
\begin{cases}
    -2 & \text{if } q,q_0 <0 \\
    0 & \text{if } q q_0 <0 \\
    2 & \text{if } q,q_0 >0
\end{cases}
\]
and the same holds for $\varsigma'$. 
Thus, it follows that 
\[ \varsigma - \varsigma' = \mathrm{sign}(q) - \mathrm{sign}(q').\]

Then, applying the formula~\eqref{eq:Ito} in two ways, and taking their difference, we obtain the following. 
\begin{align}\label{Eq2}
\begin{split}
0
&= \frac{\varsigma - \varsigma'}{8} 
+ \frac{q-q'}{p} \left( a_2(K) + 2 q_0 \,v_3(L) \right)
+ \frac{q-q'}{24p} \left( \frac{1}{qq'} - 1 \right)\\
&\quad + \frac{1}{24} \left( \left( S\left( \frac{p}{q} \right) - \frac{p}{q} \right) - \left( S\left( \frac{p}{q'} \right) - \frac{p}{q'} \right) \right) . 
\end{split}
\end{align}

By $1/q_0$-surgery on $K_0$ in $S^3$, we have an integral homology 3-sphere $\Sigma$. 
Then, the knot $K$ can be regarded as a knot in $\Sigma$ naturally. 
Let $K(p/q), K(p/q')$ denote the 3-manifold obtained from $\Sigma$ by $p/q$- and $p/q'$-surgeries on $K$, respectively. 
That is, $L(p/q,1/q_0) \cong K(p/q)$ and $L(p/q', 1/q_0) \cong K(p/q')$ hold. 
This implies that the assumption $L(p/q,1/q_0) \cong L(p/q', 1/q_0)$ is equivalent to $K(p/q) \cong K(p/q')$. 
Since $\Sigma$ is an integral homology sphere, using the surgery formula of the Casson--Gordon invariant (Equation~\eqref{CGformula}), together with $p>0$, we have 
\[
\tau (K(p/q)) - \tau (K(p/q')) = 4p\left( s(q',p) - s(q,p) \right)= 0. 
\]
On the other hand, from Equation~\eqref{eq:DedekindSymbol}, we have 
\begin{align*}
S\left(\frac{p}{q}\right) 
&= \frac{p}{q} + \frac{q}{p} + \frac{1}{pq} - 3 \mathrm{sign}(pq) - S \left( \frac{q}{p} \right) \\
&= 
\frac{p}{q} + \frac{q}{p} + \frac{1}{pq} - 3 \mathrm{sign}(pq) - 
12 \mathrm{sign}(p) s(q,p) 
\end{align*} 
and 
\[
S\left(\frac{p}{q'}\right) 
= 
\frac{p}{q'} + \frac{q'}{p} + \frac{1}{pq'} - 3 \mathrm{sign}(pq') - 
12 \mathrm{sign}(p) s(q',p).
\]
Therefore, 
it follows that 
\[
\left( S\left( \frac{p}{q} \right) - \frac{p}{q} \right)
- \left( S\left( \frac{p}{q'} \right) - \frac{p}{q'} \right) = 
\frac{q-q'}{p}\left( 1 - \frac{1}{qq'} \right) - 3 \left( \mathrm{sign}(q) - \mathrm{sign}(q') \right).
\]
Then, it follows from Equation~\eqref{Eq2} that 
\[
0
 = \frac{ q - q' }{ p } \left( a_2 (K) + 2 q_0 v_3(L) \right) 
 = \frac{ q - q' }{ p } \left( a_2 (K) - q_0 a_3(L) \right).
 \]
Consequently, it is shown that $a_2 (K) = q_0 a_3(L)$ holds. 
\end{proof}

Finally, Theorem~\ref{thm:chirally} follows from the next one.  

\begin{theorem}\label{thm6}
    Let $L = K \cup K_0$ be a 2-component link in $S^3$ with $\lk(K,K_0)=0$. 
    Let $\Sigma = K_0(p_0)$ be a rational homology sphere obtained by $p_0$-surgery on a knot $K_0$ in $S^3$ for a non-zero integer $p_0$. 
    If $p_0 \ne \dfrac{3 \sign(p_0) \pm \sqrt{96a_2(K_0)+1}}{2}$ holds, then $L(p,p_0) \not\cong - L(-p, p_0)$ for any positive integer $p$. 
\end{theorem}

\begin{proof}
Suppose that $L(p,p_0) \cong -L(-p, p_0)$, and set $\lambda_w = \lambda_w ( L(p,p_0))$, $\lambda_w' = \lambda_w ( L(- p, p_0))$. 
Then $\lambda_w' = - \lambda_w$, i.e., $\lambda_w + \lambda_w' = 0$ holds. 

We apply the formula~\eqref{eq:Ito} in two ways; apply~\eqref{eq:Ito} with $p_x = p$, $q_x=1$, $p_y=p_0$, $q_y=1$, and $p_x = -p$, $q_x = 1$, $p_y=p_0$, $q_y=1$. 
By the assumption $\ell=\lk (K,K_0) = 0$, 
the linking matrices, $A$ and $A'$, of the two manifolds $L(p,p_0)$ and $L(-p, p_0)$ are given by
\[ A = \begin{pmatrix} p & 0 \\ 0 & p_0 \end{pmatrix} \text{ and } A'=\begin{pmatrix} -p & 0 \\ 0 & p_0 \end{pmatrix}, \] 
respectively. 
Thus, their determinants, $D$ and $D'$, are given by 
\[ D = p p_0 \text{ and } D' = -p p_0 . \]
Set $\varsigma = \sigma(A)$ and $\varsigma'= \sigma(A')$. 
Then, by $p>0$ and $p_0 \ne 0$, we see that 
\[ ( \varsigma , \varsigma' ) = ( 2 , 0) \text{ or } (0,-2) . \]
Thus, it follows that 
\[ \varsigma + \varsigma' = 2 \sign(p_0) .\]
Then, by using the previous calculations, we obtain 
\begin{align*}
    \begin{split}
    \lambda_w + \lambda_w'
    &= 2 \left( \frac{\varsigma + \varsigma'}{8} 
     +\frac{2 a_2 (K_0) }{p_0} - \frac{1}{6 p_0} 
      + \frac{1}{24} \left( S\left( \frac{p}{1} \right) + S\left( \frac{p}{-1} \right) + 2 S\left( \frac{p}{1} \right) - 2 p_0 \right) \right)\\
    &= 2 \left( \frac{\sign(p_0)}{4} 
     +\frac{2 a_2 (K_0) }{p_0} - \frac{1}{6 p_0} - \frac{p_0}{12}
     \right). 
    \end{split}
\end{align*}
Thus, combining the condition $\lambda_w + \lambda_w' =0$, we conclude that
\[ p_0^2 - 3 \sign(p_0) p_0  + 2 - 24 a_2(K_0) =0. \]
This implies that $L(p,p_0) \not\cong - L(-p, p_0)$ for any positive integer $p$ if $p_0 \ne \dfrac{3\sign(p_0) \pm \sqrt{96a_2(K_0)+1}}{2}$ holds for a non-zero integer $p_0$. 
\end{proof}

\begin{proof}[Proof of Theorem~\ref{thm:chirally}]
Let $K$ be a null-homologous knot in a rational homology sphere $\Sigma$ obtained by $p_0$-Dehn surgery on a knot $K_0$ in $S^3$  for a non-zero integer $p_0$. 
By the similar argument as in Proof of Theorem~\ref{thm:rational}, we can apply Theorem~\ref{thm6}. 
Then we see that $K$ admits no integral chirally cosmetic surgeries if $p_0 \ne \dfrac{3 \sign(p_0) \pm \sqrt{96a_2(K_0)+1}}{2}$ holds. 
\end{proof}

\section{Example}\label{sec:example}

In this section, based on the arguments used in the proof of Theorem~\ref{thm3}, we give infinitely many knots admitting no integral purely cosmetic surgery pairs in certain small Seifert fibered 3-manifolds. 

Let $L(a,b) = P (2a +1, 2b, 2b)$ be the two-component pretzel link in $S^3$ as shown in Figure~\ref{fig:pretzel}. 
Here, $a, b$ are integers with $a>0$ and $b \ne 0$. 
Note that $L(a,b)$ consists of the torus knot of type $(2,2a+1)$, denoted by $T(2,2a+1)$, and the unknot denoted by $U$. 
One can easily see that the linking number $\ell$ is zero. 
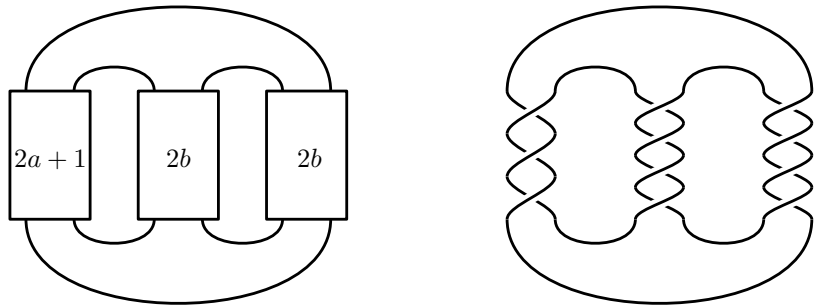
\begin{figure}[htb!]
    \centering
    \begin{overpic}[width=0.85\linewidth]{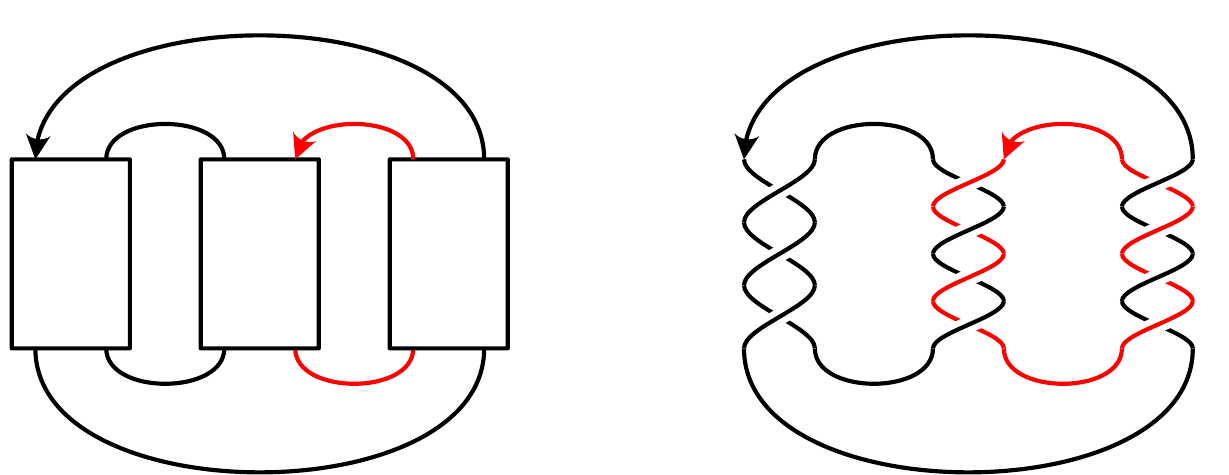}
    \put(.7,17){$2a+1$}
    \put(19.5,17){$2b$}
    \put(35.7,17){$2b$}
    \put(38,32){$K_y$}
    \put(27.5,29.5){\textcolor{red}{$K_x$}}
    \end{overpic}
    \caption{$P(2a+1, 2b, 2b)$ and $P(3,4,4)$}
    \label{fig:pretzel}
\end{figure}

Set $K_x = U$, $K_y = T(2,2a+1)$, and $L(a,b) = K_x \cup K_y$. 
Applying $p_0/q_0$-surgery on $K_y$ with $p_0/q_0 \ne 1/0, 0/1$, we have a rational homology sphere $\Sigma$ containing a knot $K$ as the image of $K_x$ after surgery. 

We show that the knot $K$ admits no integral purely cosmetic surgeries if $\dfrac{q_0}{p_0}<0$ and $b$ is sufficiently larger than $\left| \dfrac{q_0}{p_0} \right|$.  
We note that 
$\Sigma$ is an L-space if and only if $p_0/q_0 \ge 2a -1$ \cite[Example 1.6]{MotegiTeragaito14}. 
Thus, if $p_0 <0$, then $\Sigma$ is not an L-space. 
Suppose that $p/1$-surgery and $-p/1$-surgery on $K$ yield homeomorphic manifolds; that is, $K(p/1) \cong K(-p/1)$. 
By Equation~\eqref{eq:RHS1}, we have 
\begin{equation}\label{eq:pretzel1}
p^2 - \frac{3(\varsigma - \varsigma')}{2}\cdot p  + \left( 2 - 24 a_2(K_x) + \frac{24  q_0}{p_0} \,a_3(L(a,b)) \right) =0 
\end{equation}
In this case, $a_2(K_x) = 0$ since $K_x$ is trivial, and using the skein relation of the Conway polynomial, we can obtain 
\begin{equation}\label{eq:a3}
    a_3\left(L(a,b)\right) = \dfrac{-b}{6}\left( 2b^2 + 6ab + 3b +1 \right) .
\end{equation}
A precise calculation deriving Equation~\eqref{eq:a3} will be given later. 
Therefore Equation~\eqref{eq:pretzel1} is equivalent to the following quadratic equation in $p$: 
\begin{equation}\label{eq:pretzel2}
    p^2 - \dfrac{3}{2}( \varsigma - \varsigma' ) p + 2 - 4b\left( 2b^2 + 6ab + 3b +1 \right) \dfrac{q_0}{p_0} = 0. 
\end{equation}
The discriminant of LHS of Equation~\eqref{eq:pretzel2} is 
\begin{equation}\label{eq:discriminant}
    \dfrac{9(\varsigma - \varsigma')^2}{4} - 8 + 16b(2b^2 + 6ab + 3b +1) \dfrac{q_0}{p_0}. 
\end{equation}
If $\dfrac{q_0}{p_0}<0$ and $b$ is sufficiently larger than $\left| \dfrac{q_0}{p_0} \right|$, then the value of \eqref{eq:discriminant} is negative, and thus Equation~~\eqref{eq:pretzel2} has no real solution. 

Here we present a calculation deriving Equation~\eqref{eq:a3}. 
Recall the skein relation of the Conway polynomial: $\nabla_{+}(z) - \nabla_{-}(z) = z \nabla_{0}(z)$. 
Choose the orientation of $L(a,b)$ as shown in Figure~\ref{fig:pretzel}. 
Assume that $b >0$. 
Applying the skein relation to the twist box labeled $2b$, where the strands are anti-parallel (negative crossings in this case), we have 
\begin{align*}
    \nabla_{L(a,b)} =& \ \nabla_{P(2a+1,2b,2b)} \\ 
    =& \ \nabla_{P(2a+1,2(b-1),2b)} - z \nabla_{T(2,2a+2b+1)} \\ 
    =& \ \cdots \\ 
    =& \ \nabla_{P(2a+1,0,2b)} - b z \nabla_{T(2,2a+2b+1)} \\ 
    =& \ \nabla_{T(2,2a+1)} \nabla_{T(2,2b)} - b z \nabla_{T(2,2a+2b+1)} \\ 
    =& \ \left(1 + a_2(T(2,2a+1)) z^2 + \cdots \right) \left(a_1(T(2,2b)) z + a_3(T(2,2b)) z^3 + \cdots \right) \\ 
    &- b z \left(1 + a_2(T(2,2a+2b+1) z^2 + \cdots\right).  
\end{align*}
Note that $a_1(T(2,2b)) = \lk \left( T(2,2b) \right) = b$. 
Therefore, we have 
\[ a_3 (L(a,b) ) = a_3(T(2,2b))+ b a_2( T(2,2a+1) ) - b a_2( T(2,2a+2b+1) ). \]
On the other hand, we can easily show that 
\[ a_2(T(2,2k+1) ) = k(k+1)/2, \quad a_3(T(2,2k)) = (k-1)k(k+1)/6. \]
Then we obtain Equation~\eqref{eq:a3}. 
A similar calculation holds for the case where $b < 0$ or the orientation of $L(a,b)$ is changed.

\bigskip 
\thanks{Acknowledgements.} 
The authors would like to thank the anonymous referee for a careful reading and useful suggestions.

\providecommand{\bysame}{\leavevmode\hbox to3em{\hrulefill}\thinspace}
\providecommand{\MR}{\relax\ifhmode\unskip\space\fi MR }
\providecommand{\MRhref}[2]{%
  \href{http://www.ams.org/mathscinet-getitem?mr=#1}{#2}
}


\begin{thebibliography}{99}

\bibitem{BarNatanLawrence}
Dror Bar-Natan and Ruth Lawrence, \emph{A rational surgery formula for the {LMO} invariant}, Israel J. Math. \textbf{140} (2004), 29--60. \MR{2054838}

\bibitem{BleilerHodgsonWeeks}
Steven~A. Bleiler, Craig~D. Hodgson, and Jeffrey~R. Weeks, \emph{Cosmetic surgery on knots}, Proceedings of the {K}irbyfest ({B}erkeley, {CA}, 1998), Geom. Topol. Monogr., vol.~2, Geom. Topol. Publ., Coventry, 1999, pp.~23--34. \MR{1734400}

\bibitem{BoyerLines}
Steven Boyer and Daniel Lines, \emph{Surgery formulae for {C}asson's invariant and extensions to homology lens spaces}, J. Reine Angew. Math. \textbf{405} (1990), 181--220. \MR{1041002}

\bibitem{CassonGordon}
A.~J. Casson and C.~McA. Gordon, \emph{On slice knots in dimension three}, Algebraic and geometric topology ({P}roc. {S}ympos. {P}ure {M}ath., {S}tanford {U}niv., {S}tanford, {C}alif., 1976), {P}art 2, Proc. Sympos. Pure Math., vol. XXXII, Amer. Math. Soc., Providence, RI, 1978, pp.~39--53. \MR{520521}

\bibitem{daemi2024filteredinstantonhomologycosmetic}
Aliakbar Daemi, Mike~Miller Eismeier, and Tye Lidman, \emph{Filtered instanton homology and cosmetic surgery}, 2024.

\bibitem{FPS22}
David Futer, Jessica~S. Purcell, and Saul Schleimer, \emph{Effective bilipschitz bounds on drilling and filling}, Geom. Topol. \textbf{26} (2022), no.~3, 1077--1188. \MR{4466646}

\bibitem{IchiharaJongMattmanSaito}
Kazuhiro Ichihara, In~Dae Jong, Thomas~W. Mattman, and Toshio Saito, \emph{Two-bridge knots admit no purely cosmetic surgeries}, Algebr. Geom. Topol. \textbf{21} (2021), no.~5, 2411--2424. \MR{4334515}

\bibitem{IchiharaJongTaniyama}
Kazuhiro Ichihara, In~Dae Jong, and Kouki Taniyama, \emph{Achiral 1-cusped hyperbolic 3-manifolds not coming from amphicheiral null-homologous knot complements}, Lobachevskii J. Math. \textbf{39} (2018), no.~9, 1353--1361. \MR{3898025}

\bibitem{Ito2022}
Tetsuya Ito, \emph{Applications of the {C}asson-{W}alker invariant to the knot complement and the cosmetic crossing conjectures}, Geom. Dedicata \textbf{216} (2022), no.~6, Paper No. 63, 15. \MR{4475472}

\bibitem{Kirby}
Rob Kirby~(ed.), \emph{Problems in low-dimensional topology}, Geometric topology ({A}thens, {GA}, 1993), AMS/IP Stud. Adv. Math., vol. 2.2, Amer. Math. Soc., Providence, RI, 1997, pp.~35--473. \MR{1470751}

\bibitem{KotelskiyLidmanMooreWatsonZibrowius}
Artem Kotelskiy, Tye Lidman, Allison~H. Moore, Liam Watson, and Claudius Zibrowius, \emph{Cosmetic operations and {K}hovanov multicurves}, Math. Ann. \textbf{389} (2024), no.~3, 2903--2930. \MR{4753077}

\bibitem{Lescop}
Christine Lescop, \emph{Surgery formulae for finite type invariants of rational homology 3-spheres}, Algebr. Geom. Topol. \textbf{9} (2009), no.~2, 979--1047. \MR{2511138}

\bibitem{MotegiTeragaito14}
Kimihiko Motegi and Masakazu Teragaito, \emph{Left-orderable, non-{$L$}-space surgeries on knots}, Comm. Anal. Geom. \textbf{22} (2014), no.~3, 421--449. \MR{3228301}

\bibitem{StipsiczSzabo}
Andr\'{a}s~I. Stipsicz and Zolt\'{a}n Szab\'{o}, \emph{Purely cosmetic surgeries and pretzel knots}, Pacific J. Math. \textbf{313} (2021), no.~1, 195--211. \MR{4313433}

\bibitem{TaoComposite}
Ran Tao, \emph{Knots admitting purely cosmetic surgeries are prime}, Topology Appl. \textbf{322} (2022), Paper No. 108270, 11. \MR{4498360}

\bibitem{Varvarezos21}
Konstantinos Varvarezos, \emph{3-braid knots do not admit purely cosmetic surgeries}, Acta Math. Hungar. \textbf{164} (2021), no.~2, 451--457. \MR{4279346}

\bibitem{Walker}
Kevin Walker, \emph{An extension of {C}asson's invariant}, Annals of Mathematics Studies, vol. 126, Princeton University Press, Princeton, NJ, 1992. \MR{1154798}

\end{thebibliography}
\end{document}